\documentclass[12pt,reqno]{amsart}
\usepackage{Preamble}

\title{Moduli spaces of rational graphically stable curves}
\author{Andy Fry}

\begin{document}

\begin{abstract}
We use a graph to define a new stability condition for algebraic moduli spaces of rational curves. 
We characterize when the tropical compactification of the moduli space agrees with the theory of geometric tropicalization. 
The characterization statement occurs only when the graph is complete multipartite.
\end{abstract}

\maketitle


\section{Introduction}
\label{chap:intro}

A strong trend in modern algebraic geometry is the study of \emph{moduli (parameter) spaces}. 
Broadly, a moduli space parameterizes geometric objects. 
An important and well-studied moduli space is $\Mzeron$, the moduli space of smooth rational curves with $n$ marked points.
The space $\Mzeron$ is not compact, which is undesirable for algebraic geometers because of the many applications that require such a condition.
A `nice' compactification of $\Mzeron$ brings along with it a modular interpretation, that is, a compact space containing $\Mzeron$ as a dense open subset has a boundary (equal to the complement of $\Mzeron$) that parameterizes $n$-marked algebraic curves that may not be smooth. 
The most notable compactification, $\Mzeronbar$, is due to Deligne and Mumford \cite{deligne1969irreducibility}.
The boundary of their compactification is comprised of nodal curves with finite automorphism group called \emph{stable curves}. 
It is interesting to know what alternate compactifications exist and how the boundary combinatorics differs in each case.
Another important family of compactifications, $\Mzerowbar$, alters the original stability condition by assigning a weight to each marked point.
The moduli spaces of weighted stable curves were established by Hassett in the context of the log minimal model program \cite{hassett2003moduli}. 

Tropical mathematics offers tools to investigate the structure of the boundary of compact moduli spaces by relating complex algebraic varieties to piecewise linear objects.
A strength of tropical geometry is that it allows us to look at a linear skeleton of a potentially complicated variety, reducing algebro-geometric questions to those of combinatorics.
For instance, the tropical moduli space $\Mzerontrop$ is a cone complex which parameterizes leaf-labelled metric trees.
The combinatorial relation between algebraic moduli spaces and tropical moduli spaces is that the cones of $\Mzerontrop$ are in bijection with the boundary strata of $\Mzeronbar$.

Recently, we \cite{fry2019tropical} define a new family of stability conditions for tropical moduli spaces of rational marked curves determined by the combinatorics of a graph $\Gamma$, called graphic stability. 
The present document investigates how graphic stability is applied in the algebraic moduli spaces and how the algebraic and tropical moduli spaces relate to each other. 

Algebraically, we define a compactification of $\Mzeron$ using graphic stability called the \emph{moduli space of rational graphically stable pointed curves}, denoted $\MzeroGbar$.
Taking the interior, $\MzeroG$, to be smooth $\Gamma$-stable curves, these new moduli spaces have many characteristics that we would expect from a modular compactification of $\Mzeron$; namely their boundaries are divisors with simple normal crossings. 
We also construct an embedding of $\MzeroG$ into a torus using the Pl\"{u}cker embedding of the Grassmannian. 

For a smooth subvariety of a torus with a simple normal crossings compactification, the theory of geometric tropicalization relates the combinatorics of the boundary to a balanced fan in a real vector space.
Using this theory we show that the tropicalization of $\MzeroG$ is identified with a projection of the tropical moduli space $\Mzerontrop$.


This tropicalization doesn't necessarily line up with the tropical moduli space $\MzeroGtrop$.
The obstruction is a lack of injectivity in the tropicalization map. 
Specifically, the divisorial valuation map $\pi_\Gamma:\Delta(\partial\MzeroGbar)\rightarrow N_\RR$ may not be injective, this fact is highlighted in Equation~\eqref{eqn:SimpleCaseOfNonInjectivityOfPiGamma}.
The main result of this work is a classification result stating precisely when the tropical compactification of $\MzeroG$ agrees with the theory of geometric tropicalization for rational graphically stable curves.
\\

\noindent\textbf{Theorem~\ref{thm:MainMainTheorem}}
The cone complex $\MzeroGtrop$ is embedded as a balanced fan in a real vector space by $\pi_\Gamma$ if and only if $\Gamma$ is a complete multipartite graph.
For such $\Gamma$, there is a torus embedding 
$$\MzeroG\hookrightarrow T^{\binom{n}{2}-n-N}=T_\Gamma $$
whose tropicalization $\textrm{trop}(\MzeroG)$ has underlying cone complex $\MzeroGtrop$. 
Furthermore, the tropical compactification of $\MzeroG$ is $\MzeroGbar$, i.e, the closure of $\MzeroG$ in the toric variety $X(\MzeroGtrop)$ is $\MzeroGbar$.\\

The motivation for this paper comes from the theory of tropical compactifications, geometric tropicalization, and log geometry.
From work of Tevelev \cite{tevelev2007compactifications} and Gibney-Maclagan \cite{gibney2011equations} it has been shown that there is an embedding of $\M_{0,n}$ into the torus of a toric variety $X(\Sigma)$ where the tropicalization of $\M_{0,n}$ is a balanced fan $\Sigma\cong\Mzerontrop$. 
This embedding is special in the sense that the closure of $\M_{0,n}$ in $X(\Sigma)$ is $\overline{\M}_{0,n}$. 
Cavalieri et al. \cite{cavalieri2016moduli} show a similar embedding can be constructed for weighted moduli spaces when the weights are heavy/light.  
In \cite{ranganathan2017moduli}, Ranganathan-(Santos-Parker)-Wise describe radial alignments of genus 1 tropical curves and show how this extra data can be used for desingularization. 
The subdivision given by radial alignments has been studied before in \cite{ardila2006bergman} and \cite{feichtner2005matroid}, and we use a rephrasing in order to relate it to log geometry and the results of Ranganathan et al.\\

The paper is organized as follows. 
Chapter 2 discusses preliminary definitions in the algebraic (Section 2.1) and tropical (Section 2.2) settings which are necessary for this manuscript.
Section 2.3 describes the process of geometric tropicalization and briefly covers this process applied to $\Mzeron$.

Chapter 3 is composed of original work.
Section 3.1 contains a proof that $\MzeroGbar$ is not only a modular compactification of $\Mzeron$, but indeed a simple normal crossings compactification of the locus of smooth $\Gamma$-stable curves, $\MzeroG$.
To invoke geometric tropicalization, we also need a torus embedding of $\MzeroG$.
Section 3.2 begins by identifying the interior of the moduli space with the quotient of an open set of the Grassmannian, thus creating the necessary torus embedding.
We notice that the divisorial valuation map, which furnishes the combinatorics of the boundary with a fan structure, does not in general have the desired underlying cone complex, $\MzeroGtrop$.
Indeed, we achieve this compatibility only when $\Gamma$ is complete multipartite.
After the main theorem, we conclude with an example where the graph is not complete multipartite. 
In this case, the toric variety does not have enough boundary strata to contain the modular compactification. \\

\noindent{\bf Acknowledgements.} The author would like to thank Vance Blankers and Renzo Cavalieri for many helpful conversations and comments on early drafts.

\section{Preliminaries}
\label{chap:Preliminaries}



\subsection{Algebraic Moduli Spaces}

The moduli space $\Mzeron$ parameterizes isomorphism classes of smooth, genus $0$ curves with $n$ marked points. 
A point of $\Mzeron$ is an isomorphism class of $n$ ordered, distinct marked points on $\PP^1$ which we denote $(p_1,\ldots,p_n)$.
Two points $(\PP^1,p_1,\ldots,p_n),(\PP^1,q_1,\ldots,q_n)\in\Mzeron$ are equal if there is $\Phi\in\textrm{Aut}(\PP^1)$ such that $\Phi(p_i)=(q_i)$, for all $i$.
Using cross ratios, we may assign any $n$-tuple $(p_1,\ldots,p_n)$ to $(0,1,\infty,\Phi_{CR}(p_4),$ $\ldots,\Phi_{CR}(p_n))$ where $\Phi_{CR}$ is the unique automorphism of $\PP^1$ sending $p_1,$ $p_2,$ and $p_3$ to $0$, $1$, and $\infty$.
The first two nontrivial cases occur when $n=3$ and $n=4$.
As varieties, $\M_{0,3}$ is a point, as we send $(p_1,p_2,p_3)$ to $(0,1,\infty)$ and $\M_{0,4}=\PP^1\setminus\{0,1,\infty\}$ because the fourth point is free to vary as long as it doesn't coincide with the other 3 markings.
In general, this shows that $\Mzeron$ is an $n-3$ dimensional space and
$$ \Mzeron = \overbrace{\M_{0,4}\times\cdots\times \M_{0,4}}^{n-3\textrm{ times}}\setminus\{\textrm{all diagonals}\}.$$

From the $n=4$ example, we can see that $\Mzeron$ is not compact in general.
The most notable compactification, $\Mzeronbar$, is due to Deligne and Mumford which allows nodal curves with finite automorphism group; such curves are called \emph{stable} curves \cite{deligne1969irreducibility}, \cite{knudsen1983projectivity}. 

\begin{defn}\label{def:AlgebraicCurveStability}
A rational marked curve $(C,p_1,\ldots,p_n)$ is \emph{stable} if
\begin{itemize}
    \item $C$ is a connected curve of arithmetic genus 0, whose only singularities are nodes
    
    \item $(p_1,\ldots,p_n)$ are distinct points of $C\setminus \textrm{Sing}(C)$
    
    \item The only automorphism of $C$ that preserves the marked points is the identity.
\end{itemize}
\end{defn}

Stable nodal curves arise as the limit of a family of a smooth curve where a number of points collide, e.g., $p_1\mapsto p_2$.
In Figure~\ref{fig:AlgCurveAndItsDualGraph}, we see an example of a nodal curve in $\overline{\M}_{0,4}$ where the marked points $p_3,$ and $p_4$ have collided.
This curve also arises if $p_1$ and $p_2$ collide.
The \emph{dual graph} or \emph{combinatorial type} of a stable curve in $\Mzeronbar$, is defined by assigning a vertex to each component, an edge to each node, and a half-edge to each marked point, as shown in Figure~\ref{fig:AlgCurveAndItsDualGraph}.
An alternative definition of stability can be posed in terms of dual graphs.

\begin{defn}\label{def:DualGraphStability}
A rational marked curve $(C,p_1,\ldots,p_n)$ is \emph{stable} if it's dual graph is a tree where each vertex has valence greater than 2.
\end{defn}

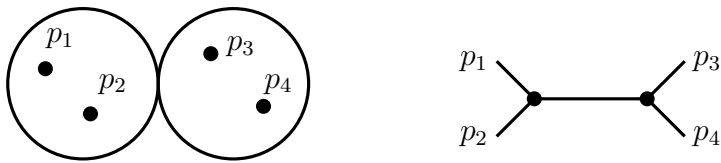
\begin{figure}[h]
\begin{tikzpicture}





\draw[very thick]  (1,.8) circle (1);
\fill (.7,1.2) circle (0.10);
\node at (1.1,1.3) {$p_3$};
\fill (1.4,.5) circle (0.10);
\node at (1.6,0.8) {$p_4$};

\draw[very thick]  (-1,.8) circle (1);
\fill (-1.5,1) circle (0.10);
\node at (-1.3,1.4) {$p_1$};
\fill (-.9,.4) circle (0.10);
\node at (-.6,.8) {$p_2$};

\begin{scope}[shift={(5,0.6)}]
\fill (0,0) circle (0.10);
\draw[very thick] (0,0) -- (-0.5,0.5); 
\node at (-0.8,0.5) {$p_1$};
\draw[very thick] (0,0) -- (-0.5,-0.5); 
\node at (-0.8,-0.5) {$p_2$};

\fill (1.5,0) circle (0.10);
\draw[very thick] (1.5,0) -- (2,0.5);
\node at (2.3,0.5) {$p_3$};
\draw[very thick] (1.5,0) -- (2,-0.5); 
\node at (2.3,-0.5) {$p_4$};

\draw[very thick] (0,0) -- (1.5,0);
\end{scope}

\end{tikzpicture}
\caption{A marked algebraic curve and it's dual graph}
\label{fig:AlgCurveAndItsDualGraph}
\end{figure} 

We define the \emph{boundary of $\Mzeronbar$} to be $\partial\Mzeronbar=\Mzeronbar\setminus\Mzeron$; it consists of all points corresponding to nodal stable curves.
We call the closure of a codimension one stratum a \emph{boundary divisor}. 
The boundary is stratified by nodal curves of a given topological type with an assignment of marks to each component.
In other words, $\partial\Mzeronbar$ is stratified by dual graphs of stable nodal pointed curves.
Dual graphs of boundary divisors partition the set of markings into two sets $I\sqcup I^c$.
We adopt the convention that the marking $1\in I^c$; therefore, a boundary divisor $D:=D_I$ is uniquely identified by its \emph{index set} $I$.

\subsection{Tropical Moduli Spaces}
We begin by introducing necessary background terminology on tropical moduli spaces.
For a more thorough survey of tropical moduli spaces, see \cite{maclagan2015introduction}.
Consider the space of genus 0, $n$-marked abstract tropical curves $\Mzerontrop$. 
Points of $\C\in\Mzerontrop$ are in bijection with metrized trees with \emph{bounded edges} having finite length and $n$ unbounded labeled edges called \emph{ends}. 
By forgetting the lengths of the bounded edges of $\C$ we get a tree with labeled ends called the \emph{combinatorial type} of $\C$. 
The space $\Mzerontrop$ naturally has the structure of a cone complex where curves of a fixed combinatorial type with $d$ bounded edges are parameterized by $\RR_{>0}^{d}$. 
We obtain $\Mzerontrop$ by gluing several copies of $\RR_{\geq0}^{n-3}$ via appropriate face morphisms, one for each trivalent combinatorial type. 

The space $\Mzerontrop$ may be embedded into a real vector space as a balanced, weighted, pure-dimensional polyhedral fan as in \cite{gathmann2009tropical}. 
We briefly recall this construction. 
A \emph{weighted fan} $(X,\omega)$ is a fan $X$ in $\RR^n$ where each top-dimensional cone $\sigma$ has a positive integer weight associated to it, denoted by $\omega(\sigma)$. 
A weighted fan is \emph{balanced} if for all cones $\tau$ of codimension one, the weighted sum of primitive normal vectors of the top-dimensional cones $\sigma_i\supset\tau$ is 0, i.e.,
$$\sum_{\sigma_i\supset\tau}\omega(\sigma_i)\cdot u_{\sigma_i/\tau}=0\in V/V_\tau$$
where $u_{\sigma_i/\tau}$ is the primitive normal vector, $V$ is the ambient real vector space, and $V_\tau$ is the smallest vector space containing the cone $\tau$. 
See \cite[Construction 2.3]{gathmann2009tropical} for a construction of the primitive normal vectors $u_{\sigma_i/\tau}$. 

For a curve $\C$, define $\dist(i,j)$ as the sum of lengths of all bounded edges between the ends marked by $i$ and $j$. Then the vector
\begin{align}
    d(\C)=(\dist(i,j))_{i<j}\in\RR^{\binom{n}{2}}/\Phi(\RR^n)=Q_n
    \label{eqn:TropicalDistanceCoords}
\end{align}
identifies $\C$ uniquely, where $\Phi:\RR^n\rightarrow\RR^{\binom{n}{2}}$ by $x\mapsto(x_i+x_j)_{i<j}$.\\



In \cite{fry2019tropical}, an alternate stability condition using a combinatorial graph is introduced for rational pointed tropical curves.

\begin{defn}\label{def:TropicalGammaStability}
The \emph{root vertex} of a stable tropical curve $\C$ is the vertex containing the end with marking 1.
A stable tropical curve $\C$ with $n$ ends is \emph{$\Gamma$-stable} if, at each non-root vertex $v$ of $\C$ with exactly one bounded edge, there exists an edge $e_{ij}\in E(\Gamma)$ where $i$ and $j$ are ends adjacent to $v$.
Define $\MzeroGtrop$ to be the parameter space of all rational $n$-marked $\Gamma$-stable abstract tropical curves. 
\end{defn}

Using graphic stability, there exists a projection map on the vector space $Q_n=\RR^{\binom{n}{2}-n}$ that forgets the coordinates corresponding to the $N$ edges removed from $K_{n-1}$ to obtain $\Gamma$, see \cite[Equations 7, 8]{fry2019tropical}.
Although there are two projections defined, Lemma 3.18 of \cite{fry2019tropical} identifies them via a linear transformation, so we will use $\projg$ to refer to both projection maps.
We will see in Lemma~\ref{lem:ClassicalAndTropicalProjectionsMatch} that $\projg$ is the tropicalization of a regular map between algebraic tori.

\subsection{Geometric Tropicalization for $\Mzeron$}
Two theories, developed simultaneously, arise when dealing with tropicalizations of subvarieties of tori: \emph{tropical compactification} and \emph{geometric tropicalization}. 
The former, introduced by Tevelev \cite{tevelev2007compactifications}, describes a situation where the tropical variety determines a good choice of compactification. 
Specifically, the tropical compactification of $U\subset\TT^r$ is its closure $\overline{U}$ in a toric variety $X(\Sigma)$ with $|\Sigma|=\textrm{trop}(U)$.
The latter, introduced by Hacking, Keel, and Tevelev \cite{hacking2009stable} and further developed by Cueto \cite{cueto2011implicitization}, explores the converse statement, how a nice compactification determines its tropicalization. 
\\

We recall some useful definitions for geometric tropicalization.
We note that geometric tropicalization can be completed with more relaxed conditions, such as replacing a smooth compactification with a normal, $\QQ$-factorial compactification and replacing simple normal crossing by combinatorial normal crossings.
For explicit details, see \cite{cueto2011implicitization}.

Let $U$ be a smooth subvariety of a torus $\TT^r$ and $Y$ be a smooth compactification containing $U$ as a dense open subvariety. 
The boundary of $Y$, $\partial Y=Y\setminus U$, is \emph{divisorial} if it is a union of codimension-1 subvarieties of $Y$.
We say $(Y,\partial Y)$ is a \emph{simple normal crossings (snc) pair} when the boundary of $Y$ behaves locally like an arrangement of coordinate hyperplanes.
In other words, $\partial Y$ is an \emph{snc divisor} if a non-empty intersection of $k$ irreducible boundary divisors is codimension $k$ and the intersection is transverse.
The \emph{boundary complex of $Y$}, $\Delta(\partial Y)$, is a simplicial complex whose vertices are in bijection with the irreducible divisors of the boundary divisor $\partial Y$, and whose $k$-cells correspond to a non-empty intersection of $k$ boundary divisors. The cells containing a face $\tau$ correspond to the boundary strata that lie in the closure of $\tau$'s stratum.

Let $\phi_1,\ldots,\phi_r\in \O^*(U)$. 
The $\phi_i$s define a morphism $\vec{\phi}$ from $U$ to a torus $\TT^r$, sending $u\in U$ to $(\phi_1(u),\ldots,\phi_r(u))$. 
When there are enough invertible functions, this map is an embedding. 
Given an irreducible boundary divisor $D\subset\partial Y$ we can compute the order of vanishing of each $\phi_i$ on $D$, $\textrm{ord}_D(\phi_i)$, yielding an $r$-dimensional integer vector $\vec{v}_D=(\textrm{ord}_D(\phi_1),\ldots,\textrm{ord}_D(\phi_r))$ living inside the cocharacter lattice of $\TT^r$, $N_{\TT^r}\subseteq N_\RR=\RR^r$.
Let $\pi:\Delta(\partial Y)\rightarrow N_\RR$ be the map defined by sending a vertex $v_i$ to $\vec{v}_{D_i}$ and extending linearly on every simplex. 
We call $\pi$ a \emph{divisorial valuation map}.
Geometric tropicalization says precisely that the support of the tropical fan is the cone over this complex and this result is independent of our choice of compactification $Y$, i.e., $\textrm{trop}(U)=\textrm{cone}(\textrm{Im}(\pi))$.
As we will see later, $\pi$ is not necessarily injective, so $\textrm{trop}(U)$ may not be the cone over $\Delta(\partial Y)$.\\

Tevelev~\cite[Theorem 5.5]{tevelev2007compactifications} first computes the tropicalization of $\Mzeron$ via geometric tropicalization by combining results of \cite{speyer2004tropical, kapranov1993chow}. 
This result is generalized by Gibney and Maclagan~\cite[Theorem 5.7]{gibney2011equations}. 
They use the fact that $\Mzeron$ can be embedded into a torus of dimension $\binom{n}{2}-n$ using the Pl{\"u}cker embedding of the Grassmannian $G(2,n)$ into $\PP^{\binom{n}{2}-1}$. 
For explicit details, see \cite{gibney2011equations} and \cite{maclagan2015introduction}. 
Comparing the algebraic Pl{\"u}cker embedding to the tropical distance coordinates we realize that the distance coordinates from Equation~\eqref{eqn:TropicalDistanceCoords} can be recovered from the tropicalization of the Pl{\"u}cker coordinates, for details see \cite[Section 3.1]{gross2016correspondence}.

\begin{exam}\label{exam:Mzero5TorusEmbeddingAndDivisorialValuation} 
For $\M_{0,5}$, we have an embedding into $T^{\binom{5}{2}-5}=T^5$.
In the boundary of $\overline{\M}_{0,5}$, there are 10 irreducible boundary divisors; they are labeled by their index sets in Figure~\ref{fig:BoundaryComplexMzeroFive}.


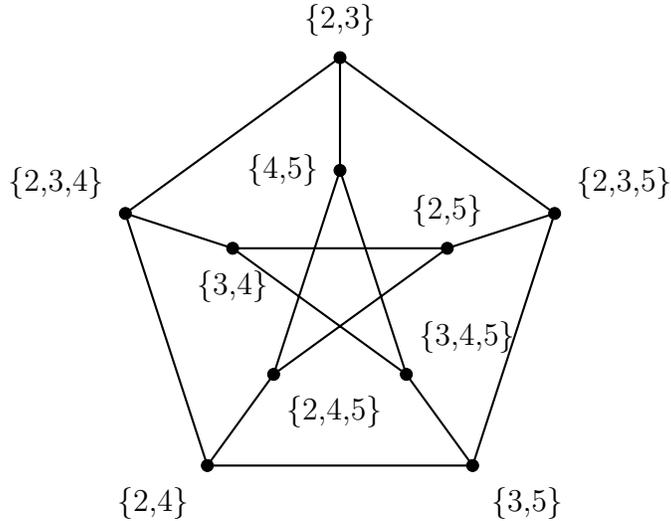
\begin{figure}[h]
    \centering
\begin{tikzpicture}[scale=0.75]
    \draw[thick] (18:4cm) -- (90:4cm) -- (162:4cm) -- (234:4cm) --
(306:4cm) -- cycle;
    \draw[thick] (18:2cm) -- (162:2cm) -- (306:2cm) -- (90:2cm) --
(234:2cm) -- cycle;


    \foreach \x in {18,90,162,234,306}
{
    \draw[thick] (\x:2cm) -- (\x:4cm);
}
    \draw[fill=black] (18:2cm) circle (3pt) node[label=90:$\{\textrm{2,5}\}$] {};
    \draw[fill=black] (90:2cm) circle (3pt) node[label=180:$\{\textrm{4,5}\}$] {};
    \draw[fill=black] (162:2cm) circle (3pt) node[label=270:$\{\textrm{3,4}\}$] {};
    \draw[fill=black] (234:2cm) circle (3pt) node[label=280:$\{\textrm{2,4,5}\}$] {};
    \draw[fill=black] (306:2cm) circle (3pt) node[label=80:$\{\textrm{3,4,5}\}$] {};
    \draw[fill=black] (18:4cm) circle (3pt) node[label=18:$\{\textrm{2,3,5}\}$] {};
    \draw[fill=black] (90:4cm) circle (3pt) node[label=90:$\{\textrm{2,3}\}$] {};
    \draw[fill=black] (162:4cm) circle (3pt) node[label=162:$\{\textrm{2,3,4}\}$] {};
    \draw[fill=black] (234:4cm) circle (3pt) node[label=234:$\{\textrm{2,4}\}$] {};
    \draw[fill=black] (306:4cm) circle (3pt) node[label=306:$\{\textrm{3,5}\}$] {};
\end{tikzpicture}
    \caption{The boundary complex of $\overline{\M}_{0,5}$ with divisors (vertices) labeled by their index set.}
    \label{fig:BoundaryComplexMzeroFive}
\end{figure}

\end{exam}

We may define $\Mzerontrop$ alternatively as the cone over $\Delta(\partial\Mzeronbar)$.
Geometric tropicalization states precisely that $\textrm{cone}(\Delta(\partial\Mzeronbar))=\textrm{trop}(\Mzeron)$. 
The following theorem, due to Tevelev and Gibney-Maclagan, states that $\Mzerontrop=\textrm{trop}(\Mzeron)$. 

\begin{thm}[\cite{tevelev2007compactifications},\cite{gibney2011equations}]
The geometric tropicalization of $\Mzeronbar$ via the embedding
$$\Mzeron\hookrightarrow T^{\binom{n}{2}-n}$$ 
gives the fan $\textrm{trop}(\Mzeron)$ whose underlying cone complex is identified with $\Mzerontrop$. Furthermore, the tropical compactification of $\Mzeron$ in the toric variety $X(\Mzerontrop)$ is $\Mzeronbar$.
\end{thm}

It follows from the previous theorem that the divisorial valuation map is injective, and thus induces a bijective map of cone complexes from $\Mzerontrop$ to $\textrm{trop}(\Mzeron)$. 
This is an important fact that we will revisit when discussing $\Gamma$-stability. \\

\section{Tropicalizing Moduli Spaces of Rational Graphically Stable Curves}
\label{chap:PhDChapter}


We define algebraic moduli spaces parameterizing rational graphically stable curves, $\MzeroGbar$, and investigate its tropicalization. 
The central result classifies \emph{all} graphically stable moduli spaces in which the tropical compactification of $\MzeroG$ agrees with the theory of geometric tropicalization.


\subsection{The Moduli Space of Rational Graphically Stable Curves}
In \cite{smyth2013towards}, Smyth gives a complete classification of all modular compactifications of $\Mzeron$ using combinatorial objects called extremal assignments (Theorem 1.9). 
In this section, we prove that $\MzeroGbar$ is a modular compactification of $\Mzeron$ by showing that $\Gamma$-stability, as in Definition~\ref{def:TropicalGammaStability}, is an extremal assignment over $\Mzeron$.
In addition, we show that the pair $(\MzeroGbar,\partial\MzeroGbar)$ is an snc pair. 
We begin with the definition of a $\Gamma$-stable curve and $\MzeroGbar$.

\begin{defn}\label{def:AlgebraicGammaStability}
Let $\Gamma$ be a simple connected graph on vertices, $2,\ldots,n$.
The \emph{root component} of a rational stable n-marked curve $(C,p_1,\ldots,p_n)$ is the component containing $p_1$. 
A rational stable $n$-marked curve $(C,p_1,\ldots,p_n)$ is \emph{$\Gamma$-stable} if, at each non-root component of $\C$ with exactly one node, there exists an edge $e_{ij}\in E(\Gamma)$ where $p_i$ and $p_j$ are points on the component.

Define $\MzeroGbar$ to be the parameter space of all rational $\Gamma$-stable $n$-marked curves with the interior $\MzeroG$ to be all smooth rational $\Gamma$-stable $n$-marked curves.
\end{defn}

Consider the assignment defined by 
\begin{align}
    \Z(C) = \{Z\subset C ~|~ |Z\cap Z^c|=1,~ p_1\not\in Z,~ \{e_{ij}\in E(\Gamma) | ~p_i,p_j\in Z\}=\emptyset \}\label{eqn:ExtremalAssignmentForGammaStability}
\end{align}
If we call a subcurve $Z\subset C$ satisfying $|Z\cap Z^c|=1$ a \emph{tail}, then the assignment $\Z$ is defined by picking out all tails of $(C,p_1,\ldots,p_n)$ not containing $p_1$ that have no edges in $\Gamma$ between vertices corresponding to the marked points on the tail. 
For the purposes of this document, we require $\Gamma$ to be a simple connected graph on $n-1$ vertices, in which case Equation~\eqref{eqn:ExtremalAssignmentForGammaStability} is an extremal assignment.

By definition, the only components contracted by $\Z$-stability are exactly those which are contracted by $\Gamma$-stability. Every tail is contracted to a point of singularity type $(0,1)$ which is a smooth point. Therefore, $\MzeroGbar=\overline{\M}_{0,n}(\Z)$ (as defined in \cite{smyth2013towards}) and so $\MzeroGbar$ is a modular compactification of $\Mzeron$.


\begin{lem}\label{lem:SimpleNormalCrossingsMzeroGBar}
The boundary $\partial\MzeroGbar=\MzeroGbar\setminus\MzeroG$ is a divisor with simple normal crossings.
\end{lem}
\begin{proof}
The boundary $\partial\MzeroGbar$ is divisorial, meaning it is a union of divisors of $\MzeroGbar$.
In addition, the boundary strata are parameterized by dual graphs. 
Each edge of a dual graph corresponds to a node in its associated complex curve, where locally each node is given by an equation $xy=t_i$ when $t_i=0$. 
Since a boundary stratum of codimension $k$ is the intersection of $k$ divisors, each divisor acts as a coordinate hyperplane $t_i=0$.
Therefore, $\partial\MzeroGbar$ behaves locally like an arrangement of coordinate hyperplanes.
\end{proof}

\subsection{Geometric Tropicalization for $\MzeroG$}

Fix $\Gamma$ to be a simple connected graph on vertices $2,\ldots,n$ containing the edge $e_{23}$.
The goal of this section is to walk through the process of geometric tropicalization for the case of $\Gamma$-stability and study the tropical compactification of $\MzeroG$. 
We begin by investigating the projection of the Pl{\"u}cker embedding of $\Mzeron$. 
Graphic stability defines a projection map that will give a torus embedding using the remaining Pl{\"u}cker coordinates. 
Next, we examine the divisorial valuation map from the boundary complex of $\MzeroGbar$ into the cocharacter lattice of the torus. 
Fixing $e_{23}\in E(\Gamma)$ prescribes a set of coordinates on the torus.
The tropicalization is a fan which coincides with the tropical moduli space $\MzeroGtrop$ if and only if $\Gamma$ is complete multipartite.\\

We may set up a torus embedding for $\MzeroG$ in the following way.
Recall that the Pl{\"u}cker embedding is given by sending a $2\times n$ matrix, representing a choice of basis for a subspace $V$, to its vector of $2\times 2$ minors called the Pl{\"u}cker coordinates.
Let $\textrm{Mat}^\Gamma(2,n)$ be the set of $2\times n$ matrices where the $ij^\textrm{th}$ minor, $x_{ij}$, is nonzero whenever $i=1$ or $e_{ij}\in E(\Gamma)$.
Let $G^\Gamma(2,n)$ be the open subspace of $G(2,n)$ given by $\textrm{Mat}^\Gamma(2,n)$; that is, the points of $G^\Gamma(2,n)$ are given by the subset of nonvanishing Pl{\"u}cker coordinates $x_{ij}$ whenever $i=1$ or $e_{ij}\in E(\Gamma)$.
Let $k$ be an algebraically closed field and consider the action of the $(n-1)$-dimensional torus $T^{n-1}=(k^*)^n/k^*$ on $\PP^{\binom{n}{2}-1}$ given by 
$$(t_1,\ldots,t_n)\cdot[x_{ij}]_{1\leq i<j\leq n}=[t_it_jx_{ij}]_{1\leq i<j\leq n}$$
The $T^{n-1}$ torus action amounts to a nonzero scaling of the columns of the $2\times n$ matrices in $G^\Gamma(2,n)$ modulo diagonal scaling and $T^{n-1}$ acts freely on $G^\Gamma(2,n)$.

An $n$-tuple of (potentially overlapping) points of $\PP^1$, $([x_1:y_1],\ldots,[x_n:y_n])$, may be encoded into a $2\times n$ matrix where each point is a column of the matrix. 
Thus, $([x_1:y_1],\ldots,[x_n:y_n])$ is sent to $(x_{12}:\cdots:x_{n-1n})\in\PP^{\binom{n}{2}-1}$ where $x_{ij}=x_iy_j-x_jy_i$.
The coordinates $x_{ij}$ are nonzero precisely when the $i$th and $j$th points are distinct; in this way, $\MzeroG$ is equal to the quotient $G^\Gamma(2,n)/T^{n-1}$.

\begin{defn}
Let $\Projg$ be the rational map from $\PP^{\binom{n}{2}-1}$ to $\PP^{\binom{n}{2}-1-N}$ dropping all the Pl{\"u}cker coordinates $x_{ij}$ for which $e_{ij}$ is not an edge of $\Gamma$, where $N$ is the number of edges removed from $K_{n-1}$ to obtain $\Gamma$. Precisely, $N=\binom{n-1}{2}-E(\Gamma)$.
\end{defn}

After applying $\Projg$, the images of the remaining Pl\"{u}cker coordinates are non-zero, meaning the image of $\textrm{Pl}\left(G^\Gamma(2,n)\right)$ via $\Projg$ lives inside a torus in $\PP^{\binom{n}{2}-1-N}$.
Quotienting by $T^{n-1}$, we get that $\MzeroG$ may be embedded into an $\left(\binom{n}{2}-n-N\right)$-dimensional torus inside $\PP^{\binom{n}{2}-1-N}$.
Recall that $\Mzeron$ lives inside an $(\binom{n}{2}-n)$-dimensional torus, $T^{\binom{n}{2}-n}\subset \PP^{\binom{n}{2}-1}$.
The projection map $\Projg$ is regular on $T^{\binom{n}{2}-n}$; in fact, the projection of $T^{\binom{n}{2}-n}$ via $\Projg$ is the torus containing $\MzeroG$. 
We prove this fact in Lemma~\ref{lem:TorusEmbeddingOfMzeroG}, and the diagram in Figure~\ref{fig:TorusEmbeddingMzeroGDiagram} summarizes the above conversation.

\begin{figure}[h]
    \centering
\begin{tikzcd}[row sep=1cm, column sep=0.85cm]
\textrm{Mat}^\Gamma(2,n) \arrow[r,] \arrow[d, two heads] & G^\Gamma(2,n) \arrow[r,"\textrm{Pl}", hook] \arrow[d, two heads] & \textrm{Pl}\left(G^\Gamma(2,n)\right) \arrow[r,symbol=\subset] \arrow[d, "\Projg"] & \PP^{\binom{n}{2}-1} \arrow[d, two heads, dashed, "\Projg"]\\
\MzeroG \arrow[r,"\cong"] & G^\Gamma(2,n)/T^{n-1} \arrow[r, hook] & T^{\binom{n}{2}-1-N}/T^{n-1} \arrow[r,symbol=\subset] & \PP^{\binom{n}{2}-1-N}
\end{tikzcd}

    \caption{Torus embedding of $\MzeroG$ via Pl{\"u}cker map.}
    \label{fig:TorusEmbeddingMzeroGDiagram}
\end{figure}

\begin{lem}\label{lem:TorusEmbeddingOfMzeroG}
The open part $\MzeroG$ can be embedded into the torus $\Projg\left(T^{\binom{n}{2}-n}\right)=T^{\binom{n}{2}-n-N}$ using the Pl{\"u}cker coordinates.
\end{lem}
\begin{proof}
Let $(\PP^1,(x_1:y_1),\ldots,(x_n:y_n))$ be a $\Gamma$-stable curve in $\MzeroG$. 
This marked curve corresponds to, up to equivalence, a point $(x_{12}:\cdots:x_{n-1n})\in\PP^{\binom{n}{2}-1}$ where $x_{ij}=x_iy_j-x_jy_i$.
Some coordinates may be zero; specifically, $x_{ij}$ is allowed to be zero when $e_{ij}\not\in E(\Gamma)$.
By definition, $\Projg(x_{12}:\cdots:x_{n-1n})=(x_{ij})_{i<j}\in\PP^{\binom{n}{2}-1-N}$ for $(i,j)=(1,j)$ where $2\leq j\leq n$ or $e_{ij}\in E(\Gamma)$.
Since each coordinate is now nonzero, $\Projg(x_{12}:\cdots:x_{n-1n})$ lies in the torus $T^{\binom{n}{2}-1-N}$.
It is not hard to see that the same $(n-1)$-dimensional torus that was quotiented out in the $\Mzeron$ case also acts on $T^{\binom{n}{2}-1-N}$.
Finally, the embedding map must be injective on $\MzeroG$ because the Pl\"{u}cker map is injective on $\MzeroG$ and the projection map is injective onto its image. 
\end{proof}

The boundary of $\MzeroGbar$ is divisorial in the same way that $\partial\Mzeronbar$ is divisorial, except that there are fewer irreducible divisors. 
The ratios $x_{ij}/x_{23}$, for $2\leq i<j\leq n$, $(i,j)\neq(2,3)$ and $e_{ij}\in E(\Gamma)$, are regular functions on $\MzeroGbar$ and act as a choice of coordinates on the torus $T^{\binom{n}{2}-n-N}$.

Define the divisorial valuation map $\pi_\Gamma:\Delta(\partial\MzeroGbar)\rightarrow N_\RR$ by assigning the vector $\vec{v}_{D_I}=(\textrm{ord}_{D_I}(x_{24}/x_{23}),\ldots,\textrm{ord}_{D_I}(x_{n-1n}/x_{23}))$ to a divisor $D_I$ where
$$
\textrm{ord}_{D_I}(x_{ij}/x_{23})=
\begin{cases} 
      1 & \{2,3\}\not\subset I,~ \{i,j\}\subset I, \textrm{ and } e_{ij}\in E(\Gamma); \\
      -1 & \{2,3\}\subset I,~ \{i,j\}\not\subset I, \textrm{ and } e_{ij}\in E(\Gamma); \\
      0 & \textrm{else}.
\end{cases}
$$
This means
\begin{align}\label{eqn:GammaDivisorialValuationMap}
\pi_\Gamma(D_I)=\vec{v}_{D_I}=
\begin{cases} 
     \displaystyle\sum_{i,j\in I}\vec{e}_{ij}  & \{2,3\}\not\subset I \textrm{ and } e_{ij}\in E(\Gamma);\\
     \displaystyle-\sum_{i,j\not\in I}\vec{e}_{ij}  & \{2,3\}\subset I \textrm{ and } e_{ij}\in E(\Gamma).
\end{cases}
\end{align}
The standard basis vectors of $T^{\binom{n}{2}-n-N}$ are given by $\vec{v}_{D_{\{i,j\}}}$, where $e_{ij}\in E(\Gamma)\setminus\{e_{23}\}$, and $\vec{v}_{D_{\{2,3\}}}=-\vec{1}$. 
For a divisor $D_I$ with $|I|\geq3$,
\begin{align}\label{eqn:AnyGammaDivisorVectorIsASumOfGammaDivisorVectors}
    \vec{v}_{D_I}&=\sum_{\substack{\{i,j\}\subset I;\\ e_{ij}\in E(\Gamma)}} \vec{v}_{D_{\{i,j\}}}.
\end{align}

\begin{lem}\label{lem:ClassicalAndTropicalProjectionsMatch}
The tropicalization of the map $\Projg$ agrees with the projection $\projg$ from \cite[Equations 7, 8]{fry2019tropical}. 
\end{lem}
\begin{proof}
A basis of $T^{\binom{n}{2}-n}$ is given by $x_{ij}/x_{23}$ for $2\leq i<j\leq n$, $(i,j)\neq(2,3)$. 
These coordinates are in bijection with divisors $D_{\{i,j\}}$.
The tropicalization of representatives of such divisors are basis elements of $\RR^{\binom{n}{2}-n}$.
Both projections $\Projg$ and $\projg$, forget coordinates that correspond to the edges deleted from $K_{n-1}$ to obtain $\Gamma$.
The discussion above confirms that the tropicalization of the basis elements of $T^{\binom{n}{2}-n-N}$ coincide with the basis elements of $\RR^{\binom{n}{2}-n-N}$.
\end{proof}

\begin{prop}\label{prop:GeometricTropicalizationOfMzeroGBar}
Using the embedding in Lemma~\ref{lem:TorusEmbeddingOfMzeroG}, the tropical variety $\textrm{trop}(\MzeroG)$ is equal to $\projg(\Mzerontrop)$.
\end{prop}
\begin{proof}
Geometric tropicalization requires a simple normal crossings compactification and a torus embedding. 
These two conditions are satisfied by Lemma~\ref{lem:SimpleNormalCrossingsMzeroGBar} and Lemma~\ref{lem:TorusEmbeddingOfMzeroG}.
By Lemma~\ref{lem:ClassicalAndTropicalProjectionsMatch} the divisorial valuations of the boundary divisors yield the rays of this fan.
Theorem 2.5 from \cite{cueto2011implicitization} states that the weight of each top-dimensional cone $\sigma\subset\textrm{trop}(\MzeroG)$ is equal to the intersection number, with multiplicity, of the divisors corresponding to the rays of $\sigma$.
A non-empty intersection of $n-3$ codimension-one curves is a single point with multiplicity 1, coinciding with the weights on $\projg(\Mzerontrop)$.
\end{proof}

From \cite{fry2019tropical}, we know that $\MzeroGtrop=\projg(\Mzerontrop)$ if and only if $\Gamma$ is a complete multipartite graph. 
Tropically, this characterization comes from studying the injectivity of a restriction morphism on graphic matroids.
Algebraically, we study the injectivity of the divisorial valuation maps.
Unlike in the $\Mzeron$ case (discussed in Section 2.3), the map $\pi_\Gamma$ may not be injective. 
There is a similar relation to Equation~\eqref{eqn:AnyGammaDivisorVectorIsASumOfGammaDivisorVectors} for $\Mzeronbar$ that we may use to demonstrate the simplest case of non-injectivity: 
consider the divisor $D_{\{i,j,k\}}$ in $\Mzeronbar$ and it's image under the divisorial valuation map
\begin{align}
    \vec{v}_{D_{\{i,j,k\}}}=\vec{v}_{D_{\{i,j\}}}+\vec{v}_{D_{\{i,k\}}}+\vec{v}_{D_{\{j,k\}}}.\label{eqn:SimpleCaseOfNonInjectivityOfPiGamma}
\end{align}
If exactly two of the vectors on the right correspond to $\Gamma$-unstable divisors, then $\pi_\Gamma$ cannot be injective.
This case does not happen when $\Gamma$ is complete multipartite. 
Example~\ref{exam:ObstructionExampleDivisorialValuation} demonstrates the failure of injectivity for $\pi_\Gamma$, while Example~\ref{exam:CompleteMultipartiteExampleDivisorialValuation} exhibits a case where $\pi_\Gamma$ is injective.

\begin{exam}\label{exam:ObstructionExampleDivisorialValuation}
Let $\Gammatilde$ be the subgraph of $K_4$ with edges $e_{35}$ and $e_{45}$ removed; see Figure~\ref{fig:GammaForExampleOfProjgFailingInjectivity}.
Then we have $\MzeroGtilde\hookrightarrow T^{\binom{5}{2}-5-2}=T^3$ with coordinates $x_{24}/x_{23},~x_{25}/x_{23},$ and $x_{34}/x_{23}$. 
In $\MzeroGtildebar$, there are 8 irreducible boundary divisors, labeled in Figure~\ref{fig:SliceOfMzeroGammaTropObstruction}. 
Comparing the cone complexes of $\MzeroGtildetrop$ and $\textrm{trop}(\MzeroGtilde)$, we can see that the cones associated to the boundary strata $D_{\{3,4\}}$, $D_{\{3,4,5\}}$, and $D_{\{3,4\}}\cap D_{\{3,4,5\}}$ in $\MzeroGtildetrop$ are all mapped to the ray given by $D_{\{3,4\}}$ in $\textrm{trop}(\MzeroGtilde)$.
Explicitly, $\pi_{\Gammatilde}:\Delta(\partial\MzeroGtildebar)\rightarrow \RR^{3}$ where the divisors have been mapped to the following primitive vectors:
\begin{align*}
    \vec{v}_{D_{\{2,4\}}}&=(1,0,0) & \vec{v}_{D_{\{2,5\}}}&=(0,1,0) &\vec{v}_{D_{\{3,4\}}}&=(0,0,1)&\\
    \vec{v}_{D_{\{2,3\}}}&=(-1,-1,-1) & \vec{v}_{D_{\{2,3,4\}}}&=(0,-1,0) &\vec{v}_{D_{\{2,3,5\}}}&=(-1,0,-1)&\\
    \vec{v}_{D_{\{2,4,5\}}}&=(1,1,0) & {\cred\vec{v}_{D_{\{3,4,5\}}}}&{\cred=(0,0,1)} & &&
\end{align*}

\begin{figure}[h]
    \centering
\begin{tikzpicture}[scale=2]
    \node (v2) at (1,1.732) [circle, draw = black, inner sep =  2pt, outer sep = 0.5pt, minimum size = 4mm, line width = 1pt] {$2$};
    \node (v3) at (2,0) [circle, draw = black, inner sep =  2pt, outer sep = 0.5pt, minimum size = 4mm, line width = 1pt] {$3$};
    \node (v4) at (0,0) [circle, draw = black, inner sep =  2pt, outer sep = 0.5pt, minimum size = 4mm, line width = 1pt] {$4$};
    \node (v5) at (1,0.577) [circle, draw = black, inner sep =  2pt, outer sep = 0.5pt, minimum size = 4mm, line width = 1pt] {$5$};
    
    \draw[line width = 1pt] (v2) -- node [midway, right] {$e_{23}$} (v3);
    \draw[line width = 1pt] (v2) -- node [midway, left] {$e_{24}$} (v4);
    \draw[line width = 1pt] (v2) -- node [midway, below left] {$e_{25}$} (v5);
    \draw[line width = 1pt] (v3) -- node [midway, below] {$e_{34}$} (v4);
\end{tikzpicture}
    \caption{The graph $\Gammatilde$ in Example~\ref{exam:ObstructionExampleDivisorialValuation}}
    \label{fig:GammaForExampleOfProjgFailingInjectivity}
\end{figure}

\begin{figure}[h]
\subfloat[A slice of the cone complex $\MzeroGtildetrop$ with rays labeled by their divisor index set.]{
    \centering
\begin{tikzpicture}[scale=0.58]
    \draw[thick] (18:4cm) -- (90:4cm) -- (162:4cm) -- (234:4cm);
    \draw[thick] (234:2cm) -- (18:2cm) -- (162:2cm);
    \draw[thick, color=red] (162:2cm) -- (306:2cm);


    \foreach \x in {18,162,234}
{
    \draw[thick] (\x:2cm) -- (\x:4cm);
}
    \draw[fill=black] (18:2cm) circle (3pt) node[label=90:$\{\textrm{2,5}\}$] {};
    \draw[fill=black] (162:2cm) circle (3pt) node[label=270:$\{\textrm{3,4}\}$] {};
    \draw[fill=black] (234:2cm) circle (3pt) node[label=280:$\{\textrm{2,4,5}\}$] {};
    \draw[fill=red, color=red] (306:2cm) circle (3pt) node[label=30:$\{\textrm{3,4,5}\}$] {};
    \draw[fill=black] (18:4cm) circle (3pt) node[label=18:$\{\textrm{2,3,5}\}$] {};
    \draw[fill=black] (90:4cm) circle (3pt) node[label=90:$\{\textrm{2,3}\}$] {};
    \draw[fill=black] (162:4cm) circle (3pt) node[label=162:$\{\textrm{2,3,4}\}$] {};
    \draw[fill=black] (234:4cm) circle (3pt) node[label=234:$\{\textrm{2,4}\}$] {};
\end{tikzpicture}
    \label{fig:SliceOfMzeroGammaTropObstruction}
}
\hfill
\subfloat[A slice of the cone complex $\textrm{trop}(\MzeroGtilde)$ with rays labeled by their divisor index set.]{
    \centering
\begin{tikzpicture}[scale=0.58]
    \draw[thick] (18:4cm) -- (90:4cm) -- (162:4cm) -- (234:4cm);
    \draw[thick] (234:2cm) -- (18:2cm) -- (162:2cm);


    \foreach \x in {18,162,234}
{
    \draw[thick] (\x:2cm) -- (\x:4cm);
}
    \draw[fill=black] (18:2cm) circle (3pt) node[label=90:$\{\textrm{2,5}\}$] {};
    \draw[fill=black] (162:2cm) circle (3pt) node[label=270:$\{\textrm{3,4}\}$] {};
    \draw[fill=black] (234:2cm) circle (3pt) node[label=280:$\{\textrm{2,4,5}\}$] {};
    \draw[fill=black] (18:4cm) circle (3pt) node[label=18:$\{\textrm{2,3,5}\}$] {};
    \draw[fill=black] (90:4cm) circle (3pt) node[label=90:$\{\textrm{2,3}\}$] {};
    \draw[fill=black] (162:4cm) circle (3pt) node[label=162:$\{\textrm{2,3,4}\}$] {};
    \draw[fill=black] (234:4cm) circle (3pt) node[label=234:$\{\textrm{2,4}\}$] {};
\end{tikzpicture}
    \label{fig:SliceOfTropOfMzeroGammaObstruction}
}
    \caption{}
\end{figure}
\end{exam}

\begin{exam}\label{exam:CompleteMultipartiteExampleDivisorialValuation}
Let $\Gamma=K_{2,2}$ be the complete bipartite graph obtained by removing edges $e_{25}$ and $e_{34}$ from $K_4$, as shown in Figure~\ref{fig:GraphOfK22}.
Then we have $\M_{0,K_{2,2}}\hookrightarrow T^{\binom{5}{2}-5-2}=T^3$ with coordinates $x_{24}/x_{23},~x_{35}/x_{23},$ and $x_{45}/x_{23}$. 
In $\overline{\M}_{0,K_{2,2}}$, there are 8 irreducible boundary divisors, labeled in Figure~\ref{fig:SilceOfTropOfMzeroGammaCompleteBipartite}. 
Explicitly, $\pi_\Gamma:\Delta(\partial\overline{\M}_{0,K_{2,2}})\rightarrow \RR^{3}$ where the divisors have been mapped to the following primitive vectors:
\begin{align*}
    \vec{v}_{D_{\{2,4\}}}&=(1,0,0) & \vec{v}_{D_{\{3,5\}}}&=(0,1,0) &\vec{v}_{D_{\{4,5\}}}&=(0,0,1)&\\
    \vec{v}_{D_{\{2,3\}}}&=(-1,-1,-1) & \vec{v}_{D_{\{2,3,4\}}}&=(0,-1,-1) &\vec{v}_{D_{\{2,3,5\}}}&=(-1,0,-1)&\\
    \vec{v}_{D_{\{2,4,5\}}}&=(1,0,1) & {\vec{v}_{D_{\{3,4,5\}}}}&{=(0,1,1)} & &&
\end{align*}

\begin{figure}[h]
    \centering
\subfloat[The graph $K_{2,2}$ in Example~\ref{exam:CompleteMultipartiteExampleDivisorialValuation}]{
\begin{tikzpicture}[scale=2]
    \node (v2) at (1,1.732) [circle, draw = black, inner sep =  2pt, outer sep = 0.5pt, minimum size = 4mm, line width = 1pt] {$2$};
    \node (v3) at (2,0) [circle, draw = black, inner sep =  2pt, outer sep = 0.5pt, minimum size = 4mm, line width = 1pt] {$3$};
    \node (v4) at (0,0) [circle, draw = black, inner sep =  2pt, outer sep = 0.5pt, minimum size = 4mm, line width = 1pt] {$4$};
    \node (v5) at (1,0.577) [circle, draw = black, inner sep =  2pt, outer sep = 0.5pt, minimum size = 4mm, line width = 1pt] {$5$};
    
    \draw[line width = 1pt] (v2) -- node [midway, right] {$e_{23}$} (v3);
    \draw[line width = 1pt] (v2) -- node [midway, left] {$e_{24}$} (v4);
    \draw[line width = 1pt] (v3) -- node [midway, above] {$e_{35}$} (v5);
    \draw[line width = 1pt] (v4) -- node [midway, below right] {$e_{45}$} (v5);
\end{tikzpicture}
    \label{fig:GraphOfK22}
}
\hfil
\subfloat[A slice of the cone complex $\textrm{trop}(\M_{0,K_{2,2}})=\M^\textrm{trop}_{0,K_{2,2}}$ with rays labeled by their divisor index set.]{
    \centering
\begin{tikzpicture}[scale=0.75]
    \draw[thick] (18:4cm) -- (90:4cm) -- (162:4cm) -- (234:4cm) --
(306:4cm) -- cycle;  
    \draw[thick] (306:2cm) -- (90:2cm) -- (234:2cm);  


    \foreach \x in {90,234,306}
{
    \draw[thick] (\x:2cm) -- (\x:4cm);
}
    \draw[fill=black] (90:2cm) circle (3pt) node[label=180:$\{\textrm{4,5}\}$] {};
    \draw[fill=black] (234:2cm) circle (3pt) node[label=280:$\{\textrm{2,4,5}\}$] {};
    \draw[fill=black] (306:2cm) circle (3pt) node[label=80:$\{\textrm{3,4,5}\}$] {};
    \draw[fill=black] (18:4cm) circle (3pt) node[label=18:$\{\textrm{2,3,5}\}$] {};
    \draw[fill=black] (90:4cm) circle (3pt) node[label=90:$\{\textrm{2,3}\}$] {};
    \draw[fill=black] (162:4cm) circle (3pt) node[label=162:$\{\textrm{2,3,4}\}$] {};
    \draw[fill=black] (234:4cm) circle (3pt) node[label=234:$\{\textrm{2,4}\}$] {};
    \draw[fill=black] (306:4cm) circle (3pt) node[label=306:$\{\textrm{3,5}\}$] {};
\end{tikzpicture}
    \label{fig:SilceOfTropOfMzeroGammaCompleteBipartite}
}
\caption{}
\end{figure}
\end{exam}

Lemma~\ref{lem:AlternateDescriptionsOfCompleteMultipartiteGraphs} gives useful characterizations of a complete multipartite graph also used in \cite{fry2019tropical}.

\begin{lem}\label{lem:AlternateDescriptionsOfCompleteMultipartiteGraphs}
Let $G$ be a graph. The following are equivalent:
\begin{enumerate}
    \item $G$ is a complete multipartite graph.
    \item If $e_{ij}$ is an edge of $G$, then for any vertex $v_k$, either $e_{ik}$ or $e_{jk}$ is an edge of $G$.
    \item There do not exist 3 vertices whose induced subgraph has exactly 1 edge. 
\end{enumerate}
\end{lem}
\begin{proof}
We can see that all three conditions express that the complement of $G$ is a disjoint union of cliques.
\end{proof}

\begin{lem}\label{lem:DivisorialValuationInjective}
The divisorial valuation map $\pi_\Gamma$ is injective if and only if $\Gamma$ is complete multipartite.
\end{lem}
\begin{proof}
We begin by proving the forwards direction by contradiction. 
Suppose $\pi_\Gamma$ is injective and $\Gamma$ is not complete multipartite.
Using Lemma~\ref{lem:AlternateDescriptionsOfCompleteMultipartiteGraphs}, fix three vertices $v_i, v_j,$ and $v_k$ where $e_{ij}\in E(\Gamma)$ but $e_{ik},e_{jk}\not\in E(\Gamma)$. 
We have the following contradiction
$$
\pi_\Gamma(D_{\{i,j,k\}})=\vec{v}_{D_{\{i,j,k\}}}=\vec{v}_{D_{\{i,j\}}}=\pi_\Gamma(D_{\{i,j\}}).
$$

For the backwards direction, assume $\Gamma$ is complete multipartite. 
Let $D_I$ and $D_J$ be two $\Gamma$-stable divisors such that $\vec{v}_{D_I}=\vec{v}_{D_J}.$
By Equation~\eqref{eqn:GammaDivisorialValuationMap} $\{2,3\}\in I$ if and only if $\{2,3\}\in J$.
If $\{2,3\}\in I,J$, then we have
$$
-\vec{1} + \sum_{\substack{\{i,j\}\subset I, ~\{i,j\}\neq\{2,3\};\\ e_{ij}\in E(\Gamma)}}\vec{e}_{ij} = -\vec{1} + \sum_{\substack{\{i,j\}\subset J, ~\{i,j\}\neq\{2,3\};\\ e_{ij}\in E(\Gamma)}}\vec{e}_{ij}.
$$
If $\{2,3\}\not\in I,J$, then we have
$$
\sum_{\substack{\{i,j\}\subset I;\\ e_{ij}\in E(\Gamma)}}\vec{e}_{ij} = \sum_{\substack{\{i,j\}\subset J;\\ e_{ij}\in E(\Gamma)}}\vec{e}_{ij}.
$$
In either case, this implies that the induced subgraphs $\Gamma_I$ and $\Gamma_J$ of $\Gamma$ have the same edge sets, $E(\Gamma_I)=E(\Gamma_J)$.
If $I\neq J$, then there exists $i\in I\setminus J$. 
But $v_i\in\Gamma$ must be isolated in $\Gamma_I$, otherwise it would be contained in an edge in $E(\Gamma_I)$ and thus $i\in J$. 
However, $\Gamma_I$ is a complete multipartite graph, so it cannot have any isolated vertices. 
Therefore, $I=J$, concluding the proof.
\end{proof}

\begin{cor}\label{cor}
The cone complex $\MzeroGtrop$ is embedded as a fan in a real vector space by $\pi_\Gamma$ if and only if $\Gamma$ is complete multipartite.
\end{cor}
\begin{proof}
The map $\pi_\Gamma$ induces a map of cone complexes from $\MzeroGtrop = \textrm{cone}(\Delta(\partial\MzeroGbar))$ to $\textrm{trop}(\MzeroG) = \textrm{cone}(\textrm{Im}(\pi_\Gamma))$ which is an isomorphism if and only if $\Gamma$ is complete multipartite by Lemma~\ref{lem:DivisorialValuationInjective}.
\end{proof}


\begin{lem}\label{lem:ForgetfulMorphismsGenerateUnits}
The units of $\O^*(\MzeroG)$ are generated by cross ratios, i.e. forgetful morphisms to $\M_{0,4}$.
\end{lem}
\begin{proof}
The space $\Mzeron$ can be viewed as the subset of $(\CC^*\setminus\{1\})^{n-3}$ minus the hyperplanes $x_i-x_j=0$. The functions which don't vanish on $\Mzeron$ are rational functions that have zeros and poles on the hyperplanes, i.e. monomials in $x_i$, $x_i-1$, and $x_i-x_j$. We can write any monomial function as a product of cross ratios:
$$\frac{(P_1-P_2)(P_3-P_4)}{(P_1-P_3)(P_2-P_4)}.$$
For $x_i$, let $P_1=x_i$, $P_2=0$, $P_3=\infty$, and $P_4=1$.\\
For $x_i-1$, let $P_1=x_i$, $P_2=1$, $P_3=\infty$, and $P_4=0$.\\
For $x_i-x_j$, take a product of $x_i$ and $P_1=x_i$, $P_2=x_j$, $P_3=0$, and $P_4=\infty$.\\
Consider the embedding of $\Mzeron$ into $\MzeroG$ in the diagram below where $\phi$ is a unit of $\O^*(\MzeroG)$.
\begin{center}
    \begin{tikzcd}[row sep=1cm, column sep=1cm]
\Mzeron \arrow[r,hook] \drar["\tilde{\phi}"']& \MzeroG \arrow[d, "\phi"] \\
& \CC^* 
\end{tikzcd}
\end{center}
From arguments above, $\tilde{\phi}$ must be a product of cross ratios. Indeed, $\phi$ is also a product of cross ratios because $\Mzeron$ is dense in $\MzeroG$. 
\end{proof}

\begin{thm}\label{thm:MainMainTheorem}
The cone complex $\MzeroGtrop$ is embedded as a balanced fan in a real vector space by $\pi_\Gamma$ if and only if $\Gamma$ is a complete multipartite graph.
For such $\Gamma$, there is a torus embedding 
$$\MzeroG\hookrightarrow T^{\binom{n}{2}-n-N}=T_\Gamma $$
whose tropicalization $\textrm{trop}(\MzeroG)$ has underlying cone complex $\MzeroGtrop$. 
Furthermore, the tropical compactification of $\MzeroG$ is $\MzeroGbar$, i.e, the closure of $\MzeroG$ in the toric variety $X(\MzeroGtrop)$ is $\MzeroGbar$.
\end{thm}




\begin{proof}
As in \cite{cavalieri2016moduli}, we wish to show the map $\MzeroGbar\rightarrow X(\MzeroGtrop)$ is an embedding. 
According to \cite[Lemma 2.6 (4) and Theorem 2.10]{hacking2009stable}, this occurs when the following two conditions hold. 
For a stratum $S$, let $\M_S$ be $\O^*(S)/{k^*}$ and $\M_{\MzeroG}^S$ be the sublattice of $\O^*(\MzeroG)/{k^*}$ generated by units having zero valuation on $S$.
\begin{enumerate}
    \item \label{itm: Main Theorem 1} For each boundary divisor $D$ containing $S$, there is a unit $u\in\O^*(\MzeroG)$ with valuation 1 on $D$ and valuation 0 on other boundary divisors containing $S$.
    \item \label{itm: Main Theorem 2} $S$ is very affine and the restriction map $\M_{\MzeroG}^S\rightarrow\M_S$ is surjective.
\end{enumerate}
We note that condition (\ref{itm: Main Theorem 1}) occurs if and only if $\Gamma$ is a complete multipartite graph, but condition (\ref{itm: Main Theorem 2}) does not force $\Gamma$ to be complete multipartite.\\

For condition (\ref{itm: Main Theorem 1}), recall that the general element of a boundary divisor $D_I$ has exactly one node and may be described by $I$, the set of marked points on a component. 
Observe that units in $\O^*(\MzeroG)$ are generated by forgetful morphisms to $\M_{0,4}$ using \emph{cross ratios} as in \cite[Section 5]{tevelev2007compactifications}. 
Such a forgetful morphism has valuation 1 on $D$ if the image of the general element of $D$ is nodal and valuation 0 on $D$ if the image of the general element of $D$ is smooth.
We show forgetful morphism with that property exists if and only if $\Gamma$ is a complete multipartite graph.

We prove the forwards direction by way of contradiction. Assume $\Gamma$ is not complete multipartite.
Using Lemma~\ref{lem:AlternateDescriptionsOfCompleteMultipartiteGraphs}, fix three vertices $v_i, v_j,$ and $v_k$ where $e_{ij}\in E(\Gamma)$ but $e_{ik},e_{jk}\not\in E(\Gamma)$. 
Consider the divisors $D_{\{i,j,k\}}$ and $D_{\{i,j\}}$ whose intersection yields the stratum $S$ whose dual graph is shown in Figure~\ref{fig:MainTheoremCounterExampleStratum}. 
Every forgetful morphism that has valuation 1 on $D_{\{i,j,k\}}$ must not forget $i$ and $j$, otherwise, the image of the general element of $D_{\{i,j,k\}}$ is smooth.
However, any such morphism also has valuation 1 on $D_{\{i,j\}}$, a contradiction.

\begin{figure}[h]
    \centering
    
\begin{tikzpicture}[scale=1.3]
\fill (0,0) circle (0.10);
\draw[very thick] (0,0) -- (-0.5,0.5); 
\node at (-0.7,0.5) {};
\draw[very thick] (0,0) -- (-0.5,0.17); 
\node at (-0.7,0.17) {};
\draw[very thick] (0,0) -- (-0.5,-0.17); 
\node at (-0.7,-0.17) {};
\draw[very thick] (0,0) -- (-0.5,-0.5); 
\node at (-0.7,-0.5) {};

\draw[very thick] (0,0) -- (1.5,0);

\fill (1.5,0) circle (0.10);
\draw[very thick] (1.5,0) -- (1.5,0.5);
\node at (1.5,0.7) {$k$};

\draw[very thick] (1.5,0) -- (3,0);

\fill (3,0) circle (0.10);
\draw[very thick] (3,0) -- (3.5,0.5); 
\node at (3.7,0.5) {$i$};
\draw[very thick] (3,0) -- (3.5,-0.5); 
\node at (3.7,-0.5) {$j$};
\end{tikzpicture}
    \caption{Dual graph of the stratum $S$ contained in $D_{\{i,j\}}$ and $D_{\{i,j,k\}}$.}
    \label{fig:MainTheoremCounterExampleStratum}
\end{figure}
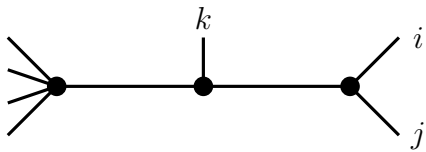

Now suppose $\Gamma$ is complete multipartite. Fix a stratum $S$ and a divisor $D_I$ containing $S$, as shown in Figure~\ref{fig:MainTheoremStratumDivisor}. 
Our goal is to find four markings, $\{a,b,c,d\}$, such that the general element of $D_I$ remains nodal and the general element of all other divisors containing $S$ become smooth in $\M_{0,4}$ on $\{a,b,c,d\}$.
We proceed by fixing $d=1$ so that $a,b,c$ correspond to vertices in $\Gamma$ and without loss of generality, let $a,b\in I$ and $c\in I^c$.
To ensure the image of the general element of $D_I$ remains nodal, we must pick $a,b$ such that $e_{ab}\in E(\Gamma)$.
Consider the components $Z$ and $Z^c$ of $S$ and $\overline{Z}$ and $\overline{Z}^c$ of $D_I$ that share a node, as illustrated by their dual graphs in Figure~\ref{fig:MainTheoremStratumDivisor}.


\begin{figure}[h]
    \centering
    
    \begin{tikzpicture}[scale=1.3]
\fill (0,0) circle (0.10);
\node at (0.2,-0.3) {$Z^c$};
\draw[very thick, dashed] (0,0) -- (-0.5,0.5); 
\draw[very thick, dashed] (0,0) -- (-0.5,0.17); 
\node at (-0.8,0) {$I^c$};
\draw[very thick] (0,0) -- (-0.5,-0.17); 
\draw[very thick] (0,0) -- (-0.5,-0.5); 

\draw[very thick] (0,0) -- (2,0);

\fill (2,0) circle (0.10);
\node at (1.8,-0.3) {$Z$};
\draw[very thick, dashed] (2,0) -- (2.5,0.5); 
\draw[very thick, dashed] (2,0) -- (2.5,0.17); 
\node at (2.8,0) {$I$};
\draw[very thick] (2,0) -- (2.5,-0.17); 
\draw[very thick] (2,0) -- (2.5,-0.5); 

\begin{scope}[shift={(5,0)}]
\fill (0,0) circle (0.10);
\node at (0.2,-0.3) {$\overline{Z}^c$};
\draw[very thick] (0,0) -- (-0.5,0.5); 
\draw[very thick] (0,0) -- (-0.5,0.17); 
\node at (-0.8,0) {$I^c$};
\draw[very thick] (0,0) -- (-0.5,-0.17); 
\draw[very thick] (0,0) -- (-0.5,-0.5); 

\draw[very thick] (0,0) -- (2,0);

\fill (2,0) circle (0.10);
\node at (1.8,-0.3) {$\overline{Z}$};
\draw[very thick] (2,0) -- (2.5,0.5); 
\draw[very thick] (2,0) -- (2.5,0.17); 
\node at (2.8,0) {$I$};
\draw[very thick] (2,0) -- (2.5,-0.17); 
\draw[very thick] (2,0) -- (2.5,-0.5); 
\end{scope}
\end{tikzpicture}
    
    \caption{Dual graphs of the stratum $S$ and divisor $D_{I}$ from the proof of Theorem~\ref{thm:MainMainTheorem} where dashed edges represent potential extra components.}
    \label{fig:MainTheoremStratumDivisor}
\end{figure}
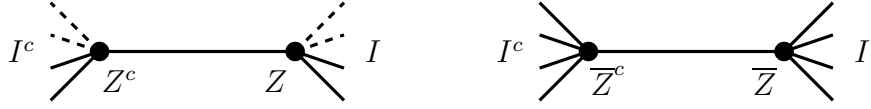


We partition $I$ and $I^c$ in the following way.
Split $S$ into connected components by separating $S$ at the nodes of $Z$ and $Z^c$.
Now, $S$ has been deconstructed into several connected components: $Z$ with its markings, $Z^c$ with its markings, and a connected component for each other node on $Z$ and $Z^c$.
Let $\lambda_I$ be the partition of $I$ given by the markings on $Z$ and the components previously attached to $Z$ and let $\lambda_{I^c}$ be the partition of $I^c$ given by the markings on $Z^c$ and the components previously attached to $Z^c$.

If $Z$ has only one node, then $\lambda_I=I$ and we need only choose $a,b\in I$ so that $e_{ab}\in E(\Gamma)$. 
Similarly, if $Z^c$ has only one node, then $\lambda_{I^c}=I^c$ and we need only choose $c\in I^c\setminus\{1\}$.
Suppose $Z$ and $Z^c$ both have more than one node.
Fix parts $A\in\lambda_I$ and $C\in\lambda_{I^c}$ so that $1\not\in C$, and let $B=I\setminus A$.
By $\Gamma$-stability, there is an edge $e_{a_1a_2}\in E(\Gamma)$ for markings $a_1,a_2\in A$. Lemma~\ref{lem:AlternateDescriptionsOfCompleteMultipartiteGraphs} says that for a marking $b\in B$, either $e_{a_1b}$ or $e_{a_2b}$ is in $E(\Gamma)$.
Fix $a\in A$ and $b\in B$ so that $e_{ab}\in E(\Gamma)$ and fix $c\in C$.

Indeed, in each of the above cases the forgetful morphism which keeps $\{1,a,b,c\}$ has valuation $1$ on $D_I$ since $e_{ab}\in E(\Gamma)$.
Let $D_J$ be any other divisor containing $S$. 
From our choices of $a,b,c$, we know $J\cap\{1,a,b,c\}\neq2$.
This means that the image of the general element $D_J$ under the forgetful morphism which keeps $\{1,a,b,c\}$ will be smooth, and thus $D_J$ has valuation $0$.

For condition (\ref{itm: Main Theorem 2}), a stratum $S$ is very affine because it can be viewed as a product of $\M_{0,\Gamma'}$s.
Each component of $S$ contains at least one node which acts as the `special' marking $1$; the marked points behave under $\Gamma'$-stability, since any subgraph of a complete multipartite graph is complete multipartite; and any extra node serves as a marked point whose vertex in $\Gamma'$ is connected to all other vertices, which keeps $\Gamma'$ as a complete multipartite graph since a fully connected vertex serves as its own independent set. 
Finally, since the boundary of $\MzeroGbar$ is a simple normal crossings divisor, as in the case of $\Mzerowbar$, the surjectivity of the restriction map follows the same proof outline as in \cite{cavalieri2016moduli}.
The local structure of $\partial\MzeroGbar$ is an intersection of coordinate hyperplanes and restricting the coordinates is surjective.
\end{proof}

Many statements remain true when $\Gamma$ is not complete multipartite. 
The geometric tropicalization of $\MzeroGbar$ using the embedding in Lemma~\ref{lem:TorusEmbeddingOfMzeroG} still equals $\projg(\MzeroGtrop)=\textrm{trop}(\MzeroG)$.
However, not all cones are mapped injectively. 
On the algebraic side, we still have a map from $\MzeroGbar$ to the toric variety $X(\textrm{trop}(\MzeroG))$, but it does not map all boundary strata injectively. 
Example~\ref{exam:ObstructionOfFansOnK4Minus2AdjacentEdgesAlgebroGeometry} highlights this observation.

\begin{exam}\label{exam:ObstructionOfFansOnK4Minus2AdjacentEdgesAlgebroGeometry}
Let $\Gammatilde$ be the subgraph of $K_4$ with edges $e_{35}$ and $e_{45}$ removed as in Example~\ref{exam:ObstructionExampleDivisorialValuation}; see Figure~\ref{fig:GammaForExampleOfProjgFailingInjectivity}.
Then Figure~\ref{fig:SliceOfMzeroGammaTropObstruction} depicts the boundary of $\MzeroGtildebar$ and a slice of $\MzeroGtildetrop$ while Figure~\ref{fig:SliceOfTropOfMzeroGammaObstruction} depicts the boundary of the closure of $\MzeroGtilde$ in $X(\textrm{trop}(\MzeroGtilde))$ and a slice of $\textrm{trop}(\MzeroGtilde)$.

There are only 8 2D cones and 7 rays in $\textrm{trop}(\MzeroGtilde)$ while the cone over $\partial\MzeroGtildebar$ has 9 2D cones and 8 rays.
This means $X(\textrm{trop}(\MzeroGtilde))$ isn't large enough to contain $\MzeroGtildebar$.
In other words, the locus of smooth curves $\MzeroGtilde$ is missing the limit as the marked points 3, 4, and 5 collide.
The modular compactification of $\MzeroGtilde$ assigns a $\PP^1$ to the limit but $X(\textrm{trop}(\MzeroGtilde))$ doesn't have enough coordinates to include a $\PP^1$.
Rather, this limit gets closed with a single point in $X(\textrm{trop}(\MzeroGtilde))$ (which is the intersection of two smooth curves where the marked points 3 and 4, and 4 and 5, have collided).

\end{exam}

\begin{rem}

Graphical stability can be viewed as a special case of the simplicial stability described by Blankers and Bozlee \cite{blankers2022compactifications} as follows. 
An \emph{independent set} of a graph is a set of vertices in a graph, no two of which are adjacent. 
The \emph{independence complex} of a graph is a simplicial complex formed by the sets of vertices in the independent sets of the graph. 
Let $\Gamma$ be a graph on $n-1$ vertices, and let $\Gamma'$ be the graph defined by adding an $n$th vertex to $\Gamma$ and edges connecting the new vertex to all other vertices. 
Then the Blankers--Bozlee simplicial compactification of $\M_{g,n}$ given by the incidence complex of $\Gamma'$ agrees with the graphical stability compactification of $\M_{g,n}$ associated to $\Gamma$.
\end{rem}

{
\bibliographystyle{plain}
\bibliography{References.bib}
}

\end{document}